\documentclass[12pt,reqno,english]{smfart}

\usepackage{amsmath,amssymb,smfthm,stmaryrd,epic,enumerate,dsfont}
\usepackage[frenchb,english]{babel}
\usepackage[all]{xy}

\usepackage{mathrsfs}
\usepackage[active]{srcltx}

\NumberTheoremsIn{subsection}\SwapTheoremNumbers

\begin{document}

\newcommand{\REMARK}[1]{\marginpar{\tiny #1}}
\newtheorem{thm}{Theorem}[subsection]
\newtheorem{lemma}[thm]{Lemma}
\newtheorem{corol}[thm]{Corollary}
\newtheorem{propo}[thm]{Proposition}
\newtheorem{defin}[thm]{Definition}
\newtheorem{Remark}[thm]{Remark}
\numberwithin{equation}{subsection}

\newtheorem{notas}[thm]{Notations}
\newtheorem{nota}[thm]{Notation}
\newtheorem{defis}[thm]{Definitions}
\newtheorem*{thm*}{Theorem}

\def\Tm{{\mathbb T}}
\def\Um{{\mathbb U}}
\def\Am{{\mathbb A}}
\def\Fm{{\mathbb F}}
\def\Mm{{\mathbb M}}
\def\Nm{{\mathbb N}}
\def\Pm{{\mathbb P}}
\def\Qm{{\mathbb Q}}
\def\Zm{{\mathbb Z}}
\def\Dm{{\mathbb D}}
\def\Cm{{\mathbb C}}
\def\Rm{{\mathbb R}}
\def\Gm{{\mathbb G}}
\def\Lm{{\mathbb L}}
\def\Km{{\mathbb K}}
\def\Om{{\mathbb O}}
\def\Em{{\mathbb E}}
\def\Xm{{\mathbb X}}

\def\BC{{\mathcal B}}
\def\QC{{\mathcal Q}}
\def\TC{{\mathcal T}}
\def\ZC{{\mathcal Z}}
\def\AC{{\mathcal A}}
\def\CC{{\mathcal C}}
\def\DC{{\mathcal D}}
\def\EC{{\mathcal E}}
\def\FC{{\mathcal F}}
\def\GC{{\mathcal G}}
\def\HC{{\mathcal H}}
\def\IC{{\mathcal I}}
\def\JC{{\mathcal J}}
\def\KC{{\mathcal K}}
\def\LC{{\mathcal L}}
\def\MC{{\mathcal M}}
\def\NC{{\mathcal N}}
\def\OC{{\mathcal O}}
\def\PC{{\mathcal P}}
\def\UC{{\mathcal U}}
\def\VC{{\mathcal V}}
\def\XC{{\mathcal X}}
\def\SC{{\mathcal S}}

\def\BF{{\mathfrak B}}
\def\AF{{\mathfrak A}}
\def\GF{{\mathfrak G}}
\def\EF{{\mathfrak E}}
\def\CF{{\mathfrak C}}
\def\DF{{\mathfrak D}}
\def\JF{{\mathfrak J}}
\def\LF{{\mathfrak L}}
\def\MF{{\mathfrak M}}
\def\NF{{\mathfrak N}}
\def\XF{{\mathfrak X}}
\def\UF{{\mathfrak U}}
\def\KF{{\mathfrak K}}
\def\FF{{\mathfrak F}}

\def \longmapright#1{\smash{\mathop{\longrightarrow}\limits^{#1}}}
\def \mapright#1{\smash{\mathop{\rightarrow}\limits^{#1}}}
\def \lexp#1#2{\kern \scriptspace \vphantom{#2}^{#1}\kern-\scriptspace#2}
\def \linf#1#2{\kern \scriptspace \vphantom{#2}_{#1}\kern-\scriptspace#2}
\def \linexp#1#2#3 {\kern \scriptspace{#3}_{#1}^{#2} \kern-\scriptspace #3}

\def \Sel {{\mathop{\mathrm{Sel}}\nolimits}}
\def \Ext{\mathop{\mathrm{Ext}}\nolimits}
\def \ad{\mathop{\mathrm{ad}}\nolimits}
\def \sh{\mathop{\mathrm{Sh}}\nolimits}
\def \irr{\mathop{\mathrm{Irr}}\nolimits}
\def \FH{\mathop{\mathrm{FH}}\nolimits}
\def \FPH{\mathop{\mathrm{FPH}}\nolimits}
\def \coh{\mathop{\mathrm{Coh}}\nolimits}
\def \res{\mathop{\mathrm{Res}}\nolimits}
\def \op{\mathop{\mathrm{op}}\nolimits}
\def \rec {\mathop{\mathrm{rec}}\nolimits}
\def \art{\mathop{\mathrm{Art}}\nolimits}
\def \vol {\mathop{\mathrm{vol}}\nolimits}
\def \cusp {\mathop{\mathrm{Cusp}}\nolimits}
\def \scusp {\mathop{\mathrm{Scusp}}\nolimits}
\def \Iw {\mathop{\mathrm{Iw}}\nolimits}
\def \JL {\mathop{\mathrm{JL}}\nolimits}
\def \speh {\mathop{\mathrm{Speh}}\nolimits}
\def \isom {\mathop{\mathrm{Isom}}\nolimits}
\def \Vect {\mathop{\mathrm{Vect}}\nolimits}
\def \groth {\mathop{\mathrm{Groth}}\nolimits}
\def \hom {\mathop{\mathrm{Hom}}\nolimits}
\def \deg {\mathop{\mathrm{deg}}\nolimits}
\def \val {\mathop{\mathrm{val}}\nolimits}
\def \det {\mathop{\mathrm{det}}\nolimits}
\def \rep {\mathop{\mathrm{Rep}}\nolimits}
\def \spec {\mathop{\mathrm{Spec}}\nolimits}
\def \fr {\mathop{\mathrm{Fr}}\nolimits}
\def \frob {\mathop{\mathrm{Frob}}\nolimits}
\def \ker {\mathop{\mathrm{Ker}}\nolimits}
\def \im {\mathop{\mathrm{Im}}\nolimits}
\def \Red {\mathop{\mathrm{Red}}\nolimits}
\def \red {\mathop{\mathrm{red}}\nolimits}
\def \aut {\mathop{\mathrm{Aut}}\nolimits}
\def \diag {\mathop{\mathrm{diag}}\nolimits}
\def \spf {\mathop{\mathrm{Spf}}\nolimits}
\def \Def {\mathop{\mathrm{Def}}\nolimits}
\def \twist {\mathop{\mathrm{Twist}}\nolimits}
\def \supp {\mathop{\mathrm{Supp}}\nolimits}
\def \Id {{\mathop{\mathrm{Id}}\nolimits}}
\def \lie {{\mathop{\mathrm{Lie}}\nolimits}}
\def \Ind{\mathop{\mathrm{Ind}}\nolimits}
\def \ind {\mathop{\mathrm{ind}}\nolimits}
\def \bad {\mathop{\mathrm{Bad}}\nolimits}
\def \top {\mathop{\mathrm{Top}}\nolimits}
\def \ker {\mathop{\mathrm{Ker}}\nolimits}
\def \coker {\mathop{\mathrm{Coker}}\nolimits}
\def \gal {{\mathop{\mathrm{Gal}}\nolimits}}
\def \Nr {{\mathop{\mathrm{Nr}}\nolimits}}
\def \rn {{\mathop{\mathrm{rn}}\nolimits}}
\def \tr {{\mathop{\mathrm{Tr~}}\nolimits}}
\def \Sp {{\mathop{\mathrm{Sp}}\nolimits}}
\def \st {{\mathop{\mathrm{St}}\nolimits}}
\def \sp{{\mathop{\mathrm{Sp}}\nolimits}}
\def \perv{\mathop{\mathrm{Perv}}\nolimits}
\def \tor {{\mathop{\mathrm{Tor}}\nolimits}}
\def \gr {{\mathop{\mathrm{gr}}\nolimits}}
\def \nilp {{\mathop{\mathrm{Nilp}}\nolimits}}
\def \obj {{\mathop{\mathrm{Obj}}\nolimits}}
\def \spl {{\mathop{\mathrm{Spl}}\nolimits}}
\def \unr {{\mathop{\mathrm{Unr}}\nolimits}}

\def \rem{{\noindent\textit{Remark:~}}}
\def \rems{{\noindent\textit{Remarques:~}}}
\def \ext {{\mathop{\mathrm{Ext}}\nolimits}}
\def \End {{\mathop{\mathrm{End}}\nolimits}}

\def\semi{\mathrel{>\!\!\!\triangleleft}}
\let \DS=\displaystyle
\def\HT{{\mathop{\mathcal{HT}}\nolimits}}

\def \hi{\HC}
\newcommand*{\tarrow}{\relbar\joinrel\mid\joinrel\twoheadrightarrow}
\newcommand*{\harrow}{\lhook\joinrel\relbar\joinrel\mid\joinrel\rightarrow}
\newcommand*{\rarrow}{\relbar\joinrel\mid\joinrel\rightarrow}
\def \coim {{\mathop{\mathrm{Coim}}\nolimits}}
\def \can {{\mathop{\mathrm{can}}\nolimits}}
\def\LFF{{\mathscr L}}

\setcounter{secnumdepth}{3} \setcounter{tocdepth}{3}

\def \Fil{\mathop{\mathrm{Fil}}\nolimits}
\def \CoFil{\mathop{\mathrm{CoFil}}\nolimits}
\def \Fill{\mathop{\mathrm{Fill}}\nolimits}
\def \CoFill{\mathop{\mathrm{CoFill}}\nolimits}
\def\SF{{\mathfrak S}}
\def\PF{{\mathfrak P}}
\def \EFil{\mathop{\mathrm{EFil}}\nolimits}
\def \EFill{\mathop{\mathrm{EFill}}\nolimits}
\def \FP{\mathop{\mathrm{FP}}\nolimits}

\let \longto=\longrightarrow
\let \oo=\infty

\let \d=\delta
\let \k=\kappa

\renewcommand{\theequation}{\arabic{section}.\arabic{subsection}.\arabic{thm}}
\newcommand{\marque}{\addtocounter{thm}{1}
{\smallskip \noindent \textit{\thethm}~---~}}

\renewcommand\atop[2]{\ensuremath{\genfrac..{0pt}{1}{#1}{#2}}}

\newcommand\atopp[2]{\genfrac{}{}{0pt}{}{#1}{#2}}

\title{$p$-stabilization in higher dimension}


\author{Boyer Pascal}

\thanks{The authors thanks the ANR for his support through the project PerCoLaTor 14-CE25.}


\begin{abstract}
Using $l$-adic completed cohomology in the context of Shimura varieties of 
Kottwitz-Harris-Taylor type attached to some fixed similitude group $G$, 
we prove, allowing to increase the levet at $l$, 
some new automorphic congruences between 
any degenerate automorphic representation with
a non degenerate one of the same weight.
\end{abstract}

\subjclass{11F70, 11F80, 11F85, 11G18, 20C08}

%

\keywords{Shimura varieties, torsion in the cohomology, maximal ideal of the Hecke algebra,
localized cohomology, galois representation}

\maketitle

\pagestyle{headings} \pagenumbering{arabic}

\tableofcontents
%
%

\section{Introduction}

In his proof of the converse to Herbrand's theorem, Ribet constructed congruences between
modular forms. More precisely, let $E_\chi$ the Eisenstein series whose constant
term is the special value $L(\chi,-1)$; if it's divisible by some prime $\lambda$
of $\overline \Qm$ over $l$, then he showed the existence of some cuspidal Hecke eigenform
whose Hecke eigenvalues are congruent to those of $E_\chi$ modulo $\lambda$. The associated
Galois representation is then irreducible but its reduction modulo $\lambda$ is not and can be chosen
non semisimple, thus giving rise to a non split extension of one dimensional Galois modules
over $\overline \Fm_l$ which can be interpreted as a non zero element in the $\chi$-component
of the class group of $\Qm(\mu_l)$.

In the automorphic setting, one starts from an automorphic representation $\pi$ of some
Levi subgroup $M$ of $G$ (eg $M=GL_1 \times GL_1$ inside $G=GL_2$) and induces it
to some automorphic representation $\Pi$ of $G(\Am)$. Then if some special value of $L(\pi)$
is divisible by $\lambda$, one can try to construct some automorphic representation $\Pi'$
whose Satake's parameters are congruent to those of $\Pi$. If we can attach Galois representations
to $\Pi$ and $\Pi'$, then we obtain as in Ribet situation, some non-split extension of
Galois modules which we can interprets as lying in some appropriate Selmer group attached to $\pi$.
This strategy have been studied by several authors: Mazur, Wiles, BellaÔche-Chenevier,
Skinner-Urban, Brown, Berger, Klosin. One of the main difficulty is the construction of $\Pi'$.

In this paper, using completed cohomology, we propose to give a relative flexible way to construct
automorphic congruences between tempered and non tempered automorphic representations
of the discrete spectrum of $GL_d(\Am_F)$ for some CM field $F$ that verify the following properties:
they are in the image of the Jacquet-Langlands correspondence described in \cite{h-t} theorem VI.1.1,
cohomological relatively to the same parameter and autodual, so that they correspond thanks to
\cite{h-t} theorem VI.2.1 to cohomological automorphic representations of some similitude
group $G$ with signatures $(1,d-1),(0,d),\cdots,(0,d)$, see \S \ref{para-geo}. 

To state the main result, 
recall that two automorphic representations are said
weakly congruent, if their Satake parameters are congruent at each place where these two 
representations are unramified, cf. \ref{defi-congruent}. 

\begin{thm*} Start from 
an irreducible automorphic $\overline \Qm_l$-representation $\Pi$ of $G(\Am)$ of some fixed weight 
$\xi$, with non trivial invariants under some open compact subgroup $I$ of $G(\Am)$, with 
degeneracy depth $s>1$, cf. definition \ref{defi-degeneracy}. 
There exists then an irreducible automorphic representation $\Pi'$ of $G(\Am)$
of the same weight $\xi$, weakly congruent to $\Pi$ modulo $l$ such that
\begin{itemize}
\item $\Pi'$ is tempered, i.e.  with degeneracy depth $1$, 

\item $\Pi'$ is of level $I'$ such that $(I')^l=I^l$.
\end{itemize}
\end{thm*}

\rem The completed cohomology was essentially introduced by Emerton in order to prove the principle 
of local-global compatibility of some expected $p$-adic Langlands correspondence 
for reductive groups as he proved it for $GL_2(\Qm_p)$ in \cite{emerton-gl2}. 
The completed cohomology groups encode in particular, besides torsion, all the automorphic congruences
between automorphic representations. 

Recall first some of the constructions of automorphic congruences already obtained by studying
torsion classes in the usual $l$-adic etale cohomology groups of Shimura varieties of 
Kottwitz-Harris-Taylor type.
\begin{itemize}
\item First, see corollary 2.9 of \cite{boyer-mrl}, to each non trivial torsion cohomology class of level $I$,
we can associate a collection $\{ \Pi(v): v \in \spl(I) \}$ indexed by some set of places of the CM field 
used to defined $G$, of non isomorphic weakly congruent irreducible automorphic representations
unramified outiside $I \cup \{ v \}$ and ramified at $v$, each of them being tempered and so
of degeneracy depth $1$.

\item In section 3 of \cite{boyer-mrl} and for some regular weight $\xi$, we constructed such torsion
classes so that we obtained the previous automorphic congruences in regular weight between 
tempered automorphic representations. Moreover we proved that each of these tempered 
representation is weakly congruent to some degenerate automorphic representation\footnote{So its 
degeneracy depth is $>1$.} of trivial weight.

\item More generally at the end of \cite{boyer-mrl}, we are able to construct automorphic congruences
between representations of different weight but without further informations about their degeneracy
depth.
\end{itemize}

Here as we allow to increase the level at the prime $l$, the result may be understood in the same
spirit as the ordinary $p$-stabilization of the Eisenstein series
$E^*_k(z):=E_k(z)-p^{k-1}E_k(pz)$ of the classical Eisenstein series $E_k(z)$ for $SL_2(\Zm)$,
which become cuspidal at one of the two cusp for $\Gamma_0(p)$.
If we think about Ribet's proof of Herbrand theorem, to be able to produce non trivial elements in some
Selmer groups, we would need to give a criterion to ensure $I'_l=I_l$ in terms of special values 
of the $L$-function associated to $\Pi$.

\section{Background}

\subsection{Geometry of KHT Shimura's varieties}
\label{para-geo}

Let $F=F^+ E$ be a CM field where $E/\Qm$ is quadratic imaginary and $F^+/\Qm$
totally real with a fixed real embedding $\tau:F^+ \hookrightarrow \Rm$. For a place $v$ of $F$,
we will denote
\begin{itemize}
\item $F_v$ the completion of $F$ at $v$,

\item $\OC_v$ the ring of integers of $F_v$,

\item $\varpi_v$ a uniformizer,

\item $q_v$ the cardinal of the residual field $\kappa(v)=\OC_v/(\varpi_v)$.
\end{itemize}
Let $B$ be a division algebra with center $F$, of dimension $d^2$ such that at every place $x$ of $F$,
either $B_x$ is split or a local division algebra and suppose $B$ provided with an involution of
second kind $*$ such that $*_{|F}$ is the complexe conjugation. For any
$\beta \in B^{*=-1}$, denote $\sharp_\beta$ the involution $x \mapsto x^{\sharp_\beta}=\beta x^*
\beta^{-1}$ and $G/\Qm$ the group of similitudes, denoted $G_\tau$ in \cite{h-t}, defined for every
$\Qm$-algebra $R$ by 
$$
G(R)  \simeq   \{ (\lambda,g) \in R^\times \times (B^{op} \otimes_\Qm R)^\times  \hbox{ such that } 
gg^{\sharp_\beta}=\lambda \}
$$
with $B^{op}=B \otimes_{F,c} F$. 
If $x$ is a place of $\Qm$ split $x=yy^c$ in $E$ then 
\addtocounter{thm}{1}
\begin{equation} \label{eq-facteur-v}
G(\Qm_x) \simeq (B_y^{op})^\times \times \Qm_x^\times \simeq \Qm_x^\times \times
\prod_{z_i} (B_{z_i}^{op})^\times,
\end{equation}
where, identifying places of $F^+$ over $x$ with places of $F$ over $y$,
$x=\prod_i z_i$ in $F^+$.

\noindent \textbf{Convention}: for $x=yy^c$ a place of $\Qm$ split in $E$ and $z$ a place of $F$ over $y$
as before, we shall make throughout the text, the following abuse of notation by denoting 
$G(F_z)$ in place of the factor $(B_z^{op})^\times$ in the formula (\ref{eq-facteur-v}).

In \cite{h-t}, the author justify the existence of some $G$ like before such that moreover
\begin{itemize}
\item if $x$ is a place of $\Qm$ non split in $E$ then $G(\Qm_x)$ is quasi split;

\item the invariants of $G(\Rm)$ are $(1,d-1)$ for the embedding $\tau$ and $(0,d)$ for the others.
\end{itemize}

As in  \cite{h-t} bottom of page 90, a compact open subgroup $U$ of $G(\Am^\oo)$ is said 
\emph{small enough}
if there exists a place $x$ such that the projection from $U^v$ to $G(\Qm_x)$ does not contain any 
element of finite order except identity.

\begin{nota}
Denote $\IC$ the set of open compact subgroup small enough of $G(\Am^\oo)$.
For $I \in \IC$, write $X_{I,\eta} \longrightarrow \spec F$ the associated
Shimura variety of Kottwitz-Harris-Taylor type.
\end{nota}

\begin{defin} \label{defi-spl}
Define $\spl$ the set of  places $v$ of $F$ such that $p_v:=v_{|\Qm} \neq l$ is split in $E$ and
$B_v^\times \simeq GL_d(F_v)$.  For each $I \in \IC$, write
$\spl(I)$ the subset of $\spl$ of places which doesn't divide the level $I$.
\end{defin}

In the sequel, $v$ will denote a place of $F$ in $\spl$. For such a place $v$ 
the scheme $X_{I,\eta}$ has a projective model $X_{I,v}$ over $\spec \OC_v$
with special fiber $X_{I,s_v}$. For $I$ going through $\IC$, the projective system $(X_{I,v})_{I\in \IC}$ 
is naturally equipped with an action of $G(\Am^\oo) \times \Zm$ such that 
$w_v$ in the Weil group $W_v$ of $F_v$ acts by $-\deg (w_v) \in \Zm$,
where $\deg=\val \circ \art^{-1}$ and $\art^{-1}:W_v^{ab} \simeq F_v^\times$ is Artin's isomorphism
which sends geometric Frobenius to uniformizers.

\begin{notas} (see \cite{boyer-invent2} \S 1.3)
For $I \in \IC$, the Newton stratification of the geometric special fiber $X_{I,\bar s_v}$ is denoted
$$X_{I,\bar s_v}=:X^{\geq 1}_{I,\bar s_v} \supset X^{\geq 2}_{I,\bar s_v} \supset \cdots \supset 
X^{\geq d}_{I,\bar s_v}$$
where $X^{=h}_{I,\bar s_v}:=X^{\geq h}_{I,\bar s_v} - X^{\geq h+1}_{I,\bar s_v}$ is an affine 
scheme\footnote{see for example \cite{ito2}}, smooth of pure dimension $d-h$ built up by the geometric 
points whose connected part of its Barsotti-Tate group is of rank $h$.
For each $1 \leq h <d$, write
$$i_{h}:X^{\geq h}_{I,\bar s_v} \hookrightarrow X^{\geq 1}_{I,\bar s_v}, \quad
j^{\geq h}: X^{=h}_{I,\bar s_v} \hookrightarrow X^{\geq h}_{I,\bar s_v},$$
and $j^{=h}=i_h \circ j^{\geq h}$.
\end{notas}
%

\subsection{Cohomology groups over $\overline \Qm_l$}

From now on, we fix a prime number $l$ unramified in $E$ and suppose that for every place $v$ of $F$ considered
after, its restriction $v_{|\Qm}$ is not equal to $l$.
Let us first recall some known facts about irreducible algebraic representations of $G$,
see for example \cite{h-t} p.97. Let $\sigma_0:E \hookrightarrow
\overline{\Qm}_l$ be a fixed embedding and et write $\Phi$ the set of embeddings 
$\sigma:F \hookrightarrow \overline \Qm_l$ whose restriction to $E$ equals $\sigma_0$.
There exists then an explicit bijection between irreducible algebraic representations $\xi$ of $G$ 
over $\overline \Qm_l$ and $(d+1)$-uple
$\bigl ( a_0, (\overrightarrow{a_\sigma})_{\sigma \in \Phi} \bigr )$
where $a_0 \in \Zm$ and for all $\sigma \in \Phi$, we have $\overrightarrow{a_\sigma}=
(a_{\sigma,1} \leq \cdots \leq a_{\sigma,d} )$.

For $K \subset \overline \Qm_l$ a finite extension of $\Qm_l$ such that the representation
$\iota^{-1} \circ \xi$ of highest weight
$\bigl ( a_0, (\overrightarrow{a_\sigma})_{\sigma \in \Phi} \bigr )$,
is defined over $K$, write $W_{\xi,K}$ the space of this representation and $W_{\xi,\OC}$
a stable lattice under the action of the maximal open compact subgroup $G(\Zm_l)$, 
where $\OC$ is the ring of integers of $K$ with uniformizer $\lambda$.

\rem if $\xi$ is supposed to be $l$-small, in the sense that for all $\sigma \in \Phi$ and all
$1 \leq i < j \leq n$ we have $0 \leq a_{\tau,j}-a_{\tau,i} < l$, then such a stable lattice is unique
up to a homothety.

\begin{nota} \label{nota-Vxi}
We will denote $V_{\xi,\OC/\lambda^n}$ the local system on $X_\IC$ as well as
$$V_{\xi,\OC}=\lim_{\atopp{\longleftarrow}{n}} V_{\xi,\OC/\lambda^n} \quad \hbox{and} \quad
V_{\xi,K}=V_{\xi,\OC} \otimes_{\OC} K.$$
For $\overline \Zm_l$ and $\overline \Qm_l$ version, we will write respectively
$V_{\xi,\overline \Zm_l}$ and $V_{\xi,\overline \Qm_l}$. 
\end{nota}

\rem the representation $\xi$ is said \emph{regular} if its parameter
$\bigl ( a_0, (\overrightarrow{a_\sigma})_{\sigma \in \Phi} \bigr )$ 
verify for all  $\sigma \in \Phi$ that $a_{\sigma,1} < \cdots < a_{\sigma,d}$.

\begin{defin}
An irreducible automorphic representation $\Pi$ is said $\xi$-cohomological if there exists
an integer $i$ such that
$$H^i \bigl ( ( \lie ~G(\Rm)) \otimes_\Rm \Cm,U,\Pi_\oo \otimes \xi^\vee \bigr ) \neq (0),$$
where $U$ is a maximal open compact subgroup modulo the center of $G(\Rm)$.
\end{defin}

\begin{nota}
For $\pi_v$ an irreducible admissible cuspidal representation of $GL_g(F_v)$ and 
$n \in \frac{1}{2} \Zm$, set 
$\pi_v \{ n \}:= \pi_v \otimes q^{-n \val \circ \det}$. Define then the Steinberg representation
$\st_s(\pi_v)$ (resp. the Speh representation $\speh_s(\pi_v)$) of 
$GL_{sg}(F_v)$, as the unique irreducible quotient (resp. subspace) 
of the standard parabolic induced representation 
$\pi_v \{ \frac{1-s}{2} \} \times \pi_v \{ \frac{3-s}{2} \} \times \cdots \times \pi_v \{ \frac{s-1}{2} \}$.
\end{nota}

A non degenerate irreducible representation of $GL_d(F_v)$ can be written as a full parabolic
induced representations $\st_{t_1}(\pi_{v,1}) \times \cdots \times \st_{t_r}(\pi_{v,r})$.
For a place $v$ such that $G(F_v) \simeq GL_d(F_v)$ in the sense of our previous convention,
the local component $\Pi_v$ of $\Pi$ at $v$ is isomorphic to some $\speh_s(\pi_v)$ where $\pi_v$
is an irreducible non degenerate representation, $s \geq 1$ is an integer and $\speh_s(\pi_v)$
is the Langlands quotients of the parabolic induced representation $\pi_v \{ \frac{1-s}{2} \} \times 
\pi_v \{ \frac{3-s}{2} \} \times \cdots \times \pi_v \{ \frac{s-1}{2} \}$. In terms of the Langlands
correspondence, $\speh_s(\pi_v)$ corresponds to $\sigma \oplus \sigma(1) \oplus \cdots \oplus \sigma(s-1)$ 
where $\sigma$ is the representation of $\gal(\bar F/F)$ associated to $\pi_v$ by the local Langlands correspondence.

\begin{defin} \label{defi-degeneracy} (cf. \cite{M-W})
For $\Pi$ an automorphic irreducible representation $\xi$-cohomological of $G(\Am)$,
then, see for example lemma 3.2 of \cite{boyer-aif},  there exists an integer $s$ called the degeneracy depth of $\Pi$,
such that through Jacquel-Langlands correspondence and base change, its associated representation of $GL_d(\Am_\Qm)$
is isobaric of the following form
$$\mu | \det |^{\frac{1-s}{2}} \boxplus \mu | \det |^{\frac{3-s}{2}} \boxplus \cdots \boxplus \mu | \det |^{\frac{s-1}{2}}$$
where $\mu$ is an irreducible cuspidal representation of $GL_{d/s}(\Am_\Qm)$.
\end{defin}

\rem if $\xi$ is a regular parameter then the depth of degeneracy of any irreducible automorphic
representation $\xi$-cohomological is necessary equal to $1$. In particular
theorem 4.3.1 of \cite{boyer-compositio} is compatible with the classical result saying that
for a regular $\xi$, the cohomology of the Shimura variety $X_I$ with coefficients
in $V_{\xi,\overline \Qm_l}$, is concentrated in middle degree.
%
%
%

\begin{nota}
For any finite level $I^l$ outside the place $l$, we denote for every $1 \leq h \leq d$:
$$H^i_{I^l,\xi}(c,h)=\lim_{\atop{ \longrightarrow}{I_l}} H^i_c(X^{=h}_{I^lI_l,\bar s_v},V_{\xi,\overline \Zm_l}[d-h]),$$
$$H^i_{I^l,\xi}(*,h)=\lim_{\atop{ \longrightarrow}{I_l}} H^i(X^{=h}_{I^lI_l,\bar s_v},j^{=h}_* V_{\xi,\overline \Zm_l}[d-h]),$$
and
$$H^i_{I^l,\xi}(h)=\lim_{\atop{ \longrightarrow}{I_l}} H^i_c(X^{\geq h}_{I^lI_l,\bar s_v},
V_{\xi,\overline \Zm_l}[d-h]),$$
where the limit is taken over all open compact subgroup $I_l$ of $G(\Qm_l)$.
\end{nota}

Over $\overline \Qm_l$, we can described these cohomology groups by taking the invariants under 
$I^l$ of
$$H^i_\xi(c,h)=\lim_{\atop{ \longrightarrow}{I^l \in \IC}} H^i_{I^l,\xi}(c,h) \quad \hbox{ and } \quad
H^i_\xi(h)=\lim_{\atop{ \longrightarrow}{I^l \in \IC}} H^i_{I^l,\xi}(h),$$
which are representations of $G(\Am^{\oo,v})$ explicitly described in \cite{boyer-compositio}.


%
%

%

\begin{propo} \cite{boyer-compositio} \S 3.6 or \cite{boyer-aif} \S 3.2 \\
Let $1 \leq h \leq d-1$ 
and $\Pi^{\oo}$ be an irreducible subquotient of $H_\xi^{i}(c,h) \otimes_{\overline \Zm_l} \overline \Qm_l$ for some $i>0$.
Then there exists a unique automorphic representation $\xi$-cohomological $\widetilde \Pi$ such that $\widetilde \Pi^{\oo,v}
\simeq \Pi^{\oo,v}$ and is of degeneracy depth $s=h+i$.
\end{propo}

\begin{propo} \cite{boyer-aif} \S 3.2 \\
Let $1 \leq h \leq d-1$ 
and $\Pi^{\oo}$ be an irreducible subquotient of $H_\xi^{i}(h)\otimes_{\overline \Zm_l} \overline \Qm_l$.
Then there exists a unique automorphic representation $\xi$-cohomological $\widetilde \Pi$ such that $\widetilde \Pi^{\oo,v}
\simeq \Pi^{\oo,v}$. Moreover 
\begin{itemize}
\item if $\widetilde \Pi$ is tempered, i.e. of degeneracy depth $1$, then $i=0$.

\item Otherwise, its degeneracy depth $s$ is $\geq h$ and $i \equiv s-h \mod 2$ with $h-s \leq i \leq s-h$.
\end{itemize}
\end{propo}

\subsection{Completed cohomology}

Given a level $I^l \in \IC$ maximal at $l$, recall that the completed cohomology groups are 
$$\widetilde{H}^i_{I^l}(V_{\xi,\OC/\lambda^n}):=\lim_{\atopp{\longrightarrow}{I_l}} 
H^i(X_{I^lI_l},V_{\xi,\OC/\lambda^n}[d-1])$$
and
$$\widetilde{H}^i_{I^l}(V_{\xi,\OC}):=\lim_{\atopp{\longleftarrow}{n}} \widetilde{H^i}_{I^l}
(V_{\xi,\OC/\lambda^n}),$$
where $\OC$ is the ring of integers of a finite extension of $\Qm_l$ on which the representation $\xi$
is defined.

\begin{nota} When $\xi=1$ is the trivial representation, we will denote
$$\widetilde{H}^i_{I^l}:=\widetilde{H}^i_{I^l}(V_{1,\OC}) \otimes_\OC \overline \Zm_l.$$
\end{nota}

\rem for $n$ fixed, there exists an open compact subgroup $I_l(n)$ such that, using the notations below
\ref{nota-Vxi}, every $I_l \subset I_l(n)$ acts trivialy on
$W_{\xi,\OC} \otimes_\OC \OC/\lambda^n$. We then deduce that the completed cohomology groups 
doesn't depend of the choice of $\xi$ in the sense where, see theorem 2.2.17 of \cite{emerton-invent}:
$$\widetilde{H}^i_{I^l}(V_{\xi,\OC})  \otimes_\OC \overline \Zm_l \simeq \widetilde{H}^i_{I^l}\otimes W_\xi$$
where $G(\Qm_l)$ acts diagonally on the right side.

\rem the choice of the " tame " level $I^l$ is harmless in the sense that, cf. \cite{emerton-invent} theorem 0.1 (ii),
for any $I^l \subset J^l$, we can recover $\widetilde H^i_{J_l}$ from $\widetilde H^i_{J_l}$ by taking invariants under $J^l/I^l$,
$$\widetilde H^i_{J_l} = \bigl ( \widetilde H^i_{I_l} \bigr )^{J_l/I_l}.$$
To recover the cohomology at finite level from completed cohomology groups, on has to use the Hochschild-Serre spectral sequence
\addtocounter{thm}{1}
\begin{equation} \label{eq-ss-completee}
E_2^{i,j}=H^i(I_l, \widetilde H^j_{I_l} \otimes V_\xi) \Rightarrow H^{i+j}(X_{I^lI_l},V_\xi[d-1]).
\end{equation}

Consider also 
$$\widehat{H}^i_{I^l}(V_{\xi,\OC})=\lim_{\atopp{\longleftarrow}{n}} \Bigl (
\lim_{\atopp{\longrightarrow}{I_l}} H^i(X_{I^lI_l},V_{\xi,\OC}[d-1]) ~/~ \lambda^n  H^i(X_{I^lI_l},V_{\xi,\OC}[d-1]) \Bigr )$$
the $l$-adic completion of $H^i_{I^l}(V_{\xi,\OC}):=
{\displaystyle \lim_{\atopp{\longrightarrow}{I_l}} } H^i(X_{I^lI_l},V_{\xi,\OC}[d-1])$
and the $l$-adic Tate module of $H^i_{I^l}(V_{\xi,\OC})$
$$T_lH^i_{I^l}(V_{\xi,\OC}):= \lim_{\atopp{\longleftarrow}{n}} H^i_{I^l}(V_{\xi,\OC}) [\lambda^n].$$
Note that the $l$-adic completion kills the $l$-divisible part while the $l$-adic Tate module
knows only about torsion. Recall then the short exact sequence
\addtocounter{thm}{1}
\begin{equation}\label{eq-sec-fond}
0 \rightarrow \widehat{H}^i_{I^l}(V_{\xi,\OC}) \longrightarrow \widetilde{H}^i_{I^l}(V_{\xi,\OC}) \longrightarrow 
T_l H^{i+1}_{I^l}(V_{\xi,\OC}) \rightarrow 0.
\end{equation}

\section{Cohomology of Harris-Taylor perverse sheaves}

We want to study $\widetilde{H}^i_{I^l}(V_{\xi,\OC}) \otimes_\OC \overline \Qm_l$ through the short exact sequence (\ref{eq-sec-fond})
so that we have to understand both
the inductive limit of the free part of the $H^i(X_{I^lI_l},V_{\xi,\OC}[d-1])$ when $I_l$ describe the subgroup of $G(\Qm_l)$ and the Tate module
$T_pH^{i+1}_{I^l}(V_{\xi,\OC})$. The strategy is to use the smooth base change theorem
at a place $v$ not above $l$ with a level
$I^l$ such that its local component at $v$ is maximal, i.e. $I_v=GL_d(\OC_v)$.

\subsection{Completed cohomology groups}

To use similar notations as in previous work, we introduce the following.

\begin{nota} For $1 \leq h \leq d$, let denote $\FC(\mathds{1},h)$ the trivial shifted local system $\overline \Zm_l[d-h]$ over
$X^{=h}_{I,\bar s_v}$. We will also use $\FC_\xi(\mathds{1},h):=\FC(\mathds{1},h) \otimes V_\xi$.
\end{nota}

Recall that for any finite level $I \in \IC$ such that $I_v=GL_d(\OC_v)$ is maximal, $X^{\geq h}_{I,\bar s_v}$ is smooth of 
dimension $d-h$.  In particular   
the shifted local system $\FC(\mathds{1},h)[d-h]$ 
is perverse  and its intermediate extension
\addtocounter{thm}{1}
\begin{equation} \label{eq-pp}
\lexp p j^{=h}_{!*} \FC(\mathds{1}_v,h)\simeq \lexp {p+} j^{=h}_{!*} \FC(\mathds{1}_v,h),
\end{equation}
where $p+$ is the $t$-structure obtained
by Grothendieck-Verdier duality, from the classical one, named $p$, is simply isomorphic to the trivial shifted local
system on $X^{\geq h}_{I,\bar s_v}$. The trivial short exact sequence
$$0 \rightarrow (\overline \Zm_l)_{X^{=h}_{I,\bar s_v}} \longrightarrow (\overline \Zm_l)_{X^{\geq h}_{I,\bar s_v}}
\longrightarrow (\overline \Zm_l)_{X^{\geq h+1}_{I,\bar s_v}} \rightarrow 0$$
can be written in terms of perverse sheaves as
\addtocounter{thm}{1}
\begin{multline}\label{eq-sec1}
0 \rightarrow \lexp p j^{=h+1}_{!*} \FC_\xi(\mathds{1}_v,h+1) 
\longrightarrow j^{=h}_! \FC_\xi(\mathds{1}_v,h) \\ 
\longrightarrow \lexp p j^{=h}_{!*} \FC_\xi(\mathds{1}_v,h) \rightarrow 0.
\end{multline}
or dually
\addtocounter{thm}{1}
\begin{multline} \label{eq-sec2}
0 \rightarrow \lexp p j^{=h}_{!*} \FC_\xi(\mathds{1}_v,h) \longrightarrow
j^{=h}_* \FC_\xi(\mathds{1}_v,h) \\ 
\longrightarrow \lexp p j^{=h+1}_{!*} \FC_\xi(\mathds{1}_v,h+1) 
\rightarrow 0.
\end{multline}

\begin{nota}
For  $1 \leq h \leq d$, we denote 
\begin{itemize}
\item $\widehat H^i_{I_l,\xi}(h)$ the $l$-adic completion of $H^i_{I^l,\xi}(h)$,

\item $T_l H^i_{I^l,\xi}(h)=\displaystyle{\lim_{\atop{\leftarrow}{n}}} H^i_{I^l,\xi}(h)[l^n]$ its $l$-adic Tate module and

\item $\widetilde H^i_{I^l}(h)$ its completed cohomology group.
\end{itemize}
\end{nota}

\begin{propo} \label{prop-completee}
Let $I^l$ a open compact subgroup of $G(\Am^{\oo,l})$ maximal at $v$. Then for any $2 \leq h \leq d$ and for any $i >0$, 
the cohomology groups $H^i_{I^l,\xi}(h)$
are divisible and free. For $h=1$, they are divisible and free for any $i>1$ and divisible for $i=1$.
\end{propo}

\begin{proof}
From the smooth base change theorem, we have $H^i(X_{I^lI_l,\bar s_v},V_{\xi,\Zm_l}) \simeq H^i(X_{I^lI_l,\bar \eta_v},V_{\xi,\Zm_l})$.
From corollary IV.2.2 of \cite{scholze-torsion}, we know that the 
\begin{itemize}
\item $\widetilde H^i(V_{\xi,\overline \Zm_l})$ are trivial for all $i > d-1$,
so that from (\ref{eq-sec-fond}), we can deduce that the $p$-adic completion
$\widehat{H}^i_{I^l}(V_{\xi,\OC})$ of $H^i_{I^l}(V_{\xi,\OC})$ is trivial, that is $H^i_{I^l}(V_{\xi,\OC})$ is divisible for every $i>0$;

\item for every $i>0$, the Tate module $T_p H^{i+1}_{I^l}(V_{\xi,\OC})$ is trivial so that for every $i>1$, the divisible module
$H^i_{I^l}(V_{\xi,\OC})$ is also free.
\end{itemize}

Then we argue by induction on $h$: assume the proposition is true for $h-1$. 
Recall that $X^{=h}_{I^lI_l,\bar s_v}$ is an affine scheme,
so that the cohomology groups $H^i(X_{I^lI_l,\bar s_v},j^{=h}_* \FC_\xi(\mathds{1}_v,h))$ 
are trivial for $i>0$. 
The long exact sequence of cohomology groups associated to the short exact sequence (\ref{eq-sec2}),
gives for every $i >0$:
$$
H^i(X_{I^lI_l},\lexp p j^{=h+1}_{!*} \FC_\xi(\mathds{1}_v,h+1)) 
\simeq  H^{i+1}(X_{I^lI_l},\lexp p j^{=h}_{!*} \FC_\xi(\mathds{1}_v,h)),
$$
which allows us to conclude par induction.
\end{proof}

\subsection{Automorphic congruences}

We want now to understand the $\widetilde H^i_{I^l}(\chi_v,h) \otimes_{\overline \Zm_l} \overline \Qm_l$ for $i<0$.
For this we first need some notations about Hecke algebras.
Let $\unr(I)$ be the union of
\begin{itemize}
\item places $q \neq l$ of $\Qm$ inert in $E$ not below a place of $\bad$ and where $I_q$ is maximal,

\item and places $w \in \spl(I)$.
\end{itemize}

\begin{nota} \label{nota-spl2}
For $I \in \IC$ a finite level, write
$$\Tm_I:=\prod_{x \in \unr(I)} \Tm_{x}$$
where for $x$ a place of $\Qm$ (resp. $x \in \spl(I)$), 
$\Tm_{x}$ is the unramified Hecke algebra of $G(\Qm_x)$ (resp. of $GL_d(F_x)$) over
$\overline \Zm_l$.
\end{nota}

\noindent \textit{Example}:  for $w \in \spl(I)$, we have
$$\Tm_{w}=\overline \Zm_l \bigl [T_{w,i}:~ i=1,\cdots,d \bigr ],$$
where $T_{w,i}$ is the characteristic function of
$$GL_d(\OC_w) \diag(\overbrace{\varpi_w,\cdots,\varpi_w}^{i}, \overbrace{1,\cdots,1}^{d-i} ) 
GL_d(\OC_w) \subset  GL_d(F_w).$$
More generally, the Satake isomorphism identifies $\Tm_x$ with $\overline \Zm_l[X^{un}(T_x)]^{W_x}$
where
\begin{itemize}
\item $T_x$ is a split torus,

\item $W_x$ is the spherical Weyl group

\item and $X^{un}(T_x)$ is the set of $\overline \Zm_l$-unramified characters of $T_x$.
\end{itemize}

Consider a fixed maximal ideal $\mathfrak m$ of $\Tm_I$ and for every $x \in \unr(I)$ let denote
$S_{\mathfrak m}(x)$ be the multi-set\footnote{A multi-set is a set with multiplicities.} 
of modulo $l$ Satake parameters at $x$ associated to $\mathfrak m$.

\noindent \textit{Example}: 
for every $w \in \spl(I)$, the multi-set of Satake parameters at $w$ corresponds to the roots of
the Hecke polynomial
$$P_{\mathfrak{m},w}(X):=\sum_{i=0}^d(-1)^i q_w^{\frac{i(i-1)}{2}} \overline{T_{w,i}} X^{d-i} \in \overline 
\Fm_l[X]$$
i.e.
$S_{\mathfrak{m}}(w) := \bigl \{ \lambda \in \Tm_I/\mathfrak m \simeq \overline \Fm_l \hbox{ such that }
P_{\mathfrak{m},w}(\lambda)=0 \bigr \}.$
For a maximal ideal $\widetilde{\mathfrak m}$ of $\Tm_{I^l} \otimes_{\overline \Zm_l} \overline \Qm_l$, we also
have the multi-set of Satake parameters 
$$S_{\widetilde{\mathfrak{m}}}(w) := \bigl \{ \lambda \in 
\Tm_I \otimes_{\overline \Zm_l} \overline \Qm_l/ \widetilde{\mathfrak m} \simeq \overline \Qm_l \hbox{ such that }
P_{\widetilde{\mathfrak{m},w}}(\lambda)=0 \bigr \}.$$

\begin{nota}
Let $\Pi$ be an irreducible automorphic representation of $G(\Am)$ of level $I$ which means here, 
that $\Pi$ has non trivial invariants under $I$ and
 for every $x \in \unr(I)$, then $\Pi_x$ is unramified. Then $\Pi$ defines a maximal ideal 
$\widetilde{\mathfrak m}(\Pi)$ of $\Tm_{I^l} \otimes_{\overline \Zm_l} \overline \Qm_l$.
\end{nota}

Let $\Pi_1$ and $\Pi_2$ be  two such automorphic representations both defined over some
finite extension $K$ of $\Qm_l$: let then $\varpi_K$ a uniformiser of the ring of integers $\OC_K$ 
of $K$. For a open compact subgroup
$I$ of $G(\Am)$ such that both $\Pi_1$ and $\Pi_2$ are of level $I$, we have their set
of Satake parameters $S_{\widetilde{\mathfrak m}(\Pi_1)}(w)$
and $S_{\widetilde{\mathfrak m}(\Pi_2)}(w)$.

\begin{defin} \label{defi-congruent}
We then say that two such automorphic representations $\Pi_1$ and $\Pi_2$ are weakly congruent
modulo $l$, if modulo $\varpi_K$ and for almost all place $w \in \unr(I)$, 
the two multi-sets $S_{\widetilde{\mathfrak m}(\Pi_1)}(w)$
and $S_{\widetilde{\mathfrak m}(\Pi_2)}(w)$ are the same.
\end{defin}

\rem $\Pi_1$ and $\Pi_2$ are weakly congruent modulo $l$ if and only if
both $\widetilde{\mathfrak m}(\Pi_1)$ and $\widetilde{\mathfrak m}(\Pi_2)$ are both contained
in the same maximal ideal $\mathfrak m$ of $\Tm_I$.

Recall that for an irreducible automorphic cohomological representation $\Pi$ with degeneracy depth equals to $s$, associated to
the maximal ideal $\widetilde{\mathfrak m}$ of $\Tm_{I^l} \otimes_{\overline \Zm_l} \overline \Qm_l$ where $\Pi^{I^l} \neq (0)$,
then for any $w \in \spl(I)$, the set $S_{\widetilde{\mathfrak{m}}}(w)$ 
\begin{itemize}
\item can be written as a disjoint union of $g$ segments
$\{ \alpha,q_w \alpha, \cdots, q_w^{s-1} \alpha \}$ of length $s$,

\item and does not contain any segment of length strictly greater than $s$.
\end{itemize}
We will say that $\widetilde{\mathfrak m}$ is of degeneracy depth $s$.

%
%
%
%

\subsection{Proof of the main result}

Fix now an irreducible automorphic $\overline \Qm_l$-representation $\Pi$ of $G(\Am)$ of some
fixed weight $\xi$, with level $I$ and degeneracy depth $s>1$. Let $\widetilde{\mathfrak m}(\Pi)$
be the prime ideal of $\Tm_I$ associated to $\Pi$ and $\mathfrak m$ the maximal ideal 
of $\Tm_I$ containing $\widetilde{\mathfrak m}(\Pi)$. 

\marque \textit{Hypothesis}: through all this section, we then suppose by absurdity, that the conclusion of the theorem in
the introduction is false, i.e. there is no tempered irreducible representation $\Pi'$ of weight $\xi$
which is weakly congruent to $\Pi$, and with non trivial
invariants under $I'=I^l I'_l$ whatever is the choice of $I_l$.

\begin{propo}
Under the previous hypothesis, for all $1 \leq t \leq d$ and every $i$, the localized cohomology
groups $H^i_{I^l,\xi}(h)_{\mathfrak m}$ are free.
\end{propo}

\begin{proof}
We follow, in the spirit, the proof of the main theorem 4.7 of \cite{boyer-imj}. We use the short exact
sequence (\ref{eq-sec1}).
Over $\overline \Qm_l$, the only non trivial map between 
\begin{multline} \label{eq-ss0}
H^i(X_{I^l I_l,\bar s_v}, \lexp p j^{=t+1}_{!*} \FC_\xi(\mathds{1}_v,t+1))
\longrightarrow H^i(X_{I^lI_l,\bar s_v}, j^{=h}_! \FC_\xi(\mathds{1}_v,h))
\end{multline}
is for $i=0$ and corresponds to tempered representations. In particular, under our hypothesis,
all these maps are trivial after $\mathfrak m$-localisation.

Argue then by induction on $t$ from $d$ to $1$. The case $t=d$ is trivial as 
$X^{\geq d}_{I^lI_l,\bar s_v}$ is $0$-dimensional. Suppose now the result true for $t > t_0$.
As $X^{=t_0}_{I^lI_l,\bar s_v}$ is affine, then $H^i(X_{I^lI_l,\bar s_v},j^{=h}_! \FC_\xi(\mathds{1}_v,h))$ 
is trivial for every $i<0$ and free for $i=0$.
In particular the result is true for every $i<0$ and for $i=0$, it's true because the map in
(\ref{eq-ss0}) for $i=0$ and $t=t_0$, is, after $\mathfrak m$-localisation, trivial. 
We then deduce the result for $i>0$, using
Grothendieck-Verdier duality and the isomorphism (\ref{eq-pp}).
\end{proof}

\begin{defin}
Let $s_0 \geq s$ maximal such that there exists an irreducible automorphic representation $\Pi'$ of weight $\xi$ weakly
congruent to $\Pi$ and with level of the form $I'=I^lI'_l$ for some compact open subgroup $I'_l$ of $G(\Qm_l)$.
\end{defin}

\begin{corol}
Under our hypothesis, 
$\widetilde H^i_{I^l}(t)_{\mathfrak m}\simeq \widehat H^i_{I^l}(t)_{\mathfrak m}$ 
verify the following properties:
\begin{itemize}
\item It's trivial for every $t>s_0$.

\item For $t \leq s_0$, if it's non trivial then $i \geq t-s_0$, and it's non trivial for $i=t-s_0$.
\end{itemize}
\end{corol}

\rem recall from proposition \ref{prop-completee}, that these completed cohomology groups
are trivial for $i>0$.

\begin{proof}
The isomorphism follows from the previous proposition and the short exact sequence
(\ref{eq-sec-fond}).
If $t>s_0$, the result follows from the fact that for every $i$, we have 
$H^i_{I^l,\xi}(t)_{\mathfrak m}=(0)$.
For $t=s_0$, the same argument tells us that $\widetilde H^i_{I^l}(s_0)_{\mathfrak m}=(0)$
for all $i \neq 0$. If $\widetilde H^0_{I^l}(s_0)_{\mathfrak m}$ were trivial, then from the
spectral sequence (\ref{eq-ss-completee}), then all the 
$H^i(X^{s_0}_{I^lI_l},V_{\xi,\overline \Zm_l})_{\mathfrak m}$ would be trivial, which is not the case
for $i=0$.

Suppose now the result true for $t > t_0$. Recall from the previous proof, that the long exact sequence
associated to (\ref{eq-sec1}), gives for all $i<0$:
$$H^i(X_{I^lI_l,\bar s_v},\lexp p j^{=t}_{!*} \FC_\xi(\mathds{1}_v,t))_{\mathfrak m} \simeq
H^{i+1}(X_{I^lI_l,\bar s_v},\lexp p j^{=t+1}_{!*} \FC_\xi(\mathds{1}_v,t+1))_{\mathfrak m},$$
so we conclude by induction.
\end{proof}

Then we look at the cohomology groups $H^i_{I^l,\xi}(*,s_0-1)_{\mathfrak m}$ through the short exact sequence (\ref{eq-sec2})
where we recall that, thanks to our absurd hypothesis, the $H^i_{I^l,\xi}(h)_{\mathfrak m}$ are free.
In particular we have
$$\cdots \rightarrow H^0_{I^l,\xi}(s_0)_{\mathfrak m} \longrightarrow H^1_{I^l,\xi}(s_0-1)_{\mathfrak m} \longrightarrow 
H^1_{I^l,\xi}(*,s_0-1)_{\mathfrak m} \rightarrow \cdots$$
where, as over $\overline \Qm_l$,  
$$H^0_{I^l,\xi}(s_0)_{\mathfrak m} \otimes_{\overline \Zm_l} \overline \Qm_l \hookrightarrow H^1_{I^l,\xi}(s_0-1)_{\mathfrak m}
\otimes_{\overline \Zm_l} \overline \Qm_l$$
the first map $H^0_{I^l,\xi}(s_0)_{\mathfrak m} \longrightarrow H^1_{I^l,\xi}(s_0-1)_{\mathfrak m}$ is injective.
But we show that
\begin{itemize}
\item $H^0_{I^l,\xi}(s_0)_{\mathfrak m}$ is not divisible, cf. previous corollary;

\item $H^1_{I^l,\xi}(s_0-1)_{\mathfrak m}$ is divisible, cf. proposition \ref{prop-completee}.
\end{itemize}
We then deduce that the torsion of $H^1_{I^l,\xi}(*,s_0-1)_{\mathfrak m}$ is non trivial which is impossible because, as
$X^{=s_0-1}_{I^lI_l,\bar s_v}$ is affine, then  
$$H^i(X^{\geq s_0-1}_{I^lI_l,\bar s_v},j^{\geq s_0-1}_* \FC_\xi(\mathds{1},s_0-1))=(0) \quad \forall i>0.$$
We then conclude that our previous hypothesis was absurd and so the main theorem is proved.

\bibliographystyle{plain}
\bibliography{bib-ok}

\address{Universit\'e Paris 13, Sorbonne Paris Cit\'e \\
LAGA, CNRS, UMR 7539\\ 
F-93430, Villetaneuse (France) \\
PerCoLaTor: ANR-14-CE25}

\email{boyer@math.univ-paris13.fr}

\end{document}